\documentclass[12pt]{amsart}
\usepackage{etex}

\usepackage{amsthm,amssymb,enumitem,mathtools}

\usepackage{cite}

\usepackage{tikz}

\usepackage{scalerel}

\usepackage{cancel}

\usepackage{framed}

\usetikzlibrary{arrows,shapes,automata,backgrounds,decorations,petri,positioning}
\usetikzlibrary[calc,intersections,through,backgrounds,arrows,decorations.pathmorphing,decorations.markings]

\tikzstyle{edge} = [fill,opacity=.5,fill opacity=.5,line cap=round, line join=round, line width=50pt]

\pgfdeclarelayer{background}
\pgfsetlayers{background,main}

\usepackage[margin=1in,letterpaper,portrait]{geometry}

\usepackage{amsthm}

\usepackage[latin1]{inputenc}
\usepackage{subfigure}
\usepackage{color}
\usepackage{amsmath}
\usepackage{amsthm}
\usepackage{amstext}
\usepackage{amssymb}
\usepackage{amsfonts}
\usepackage{graphicx}
\usepackage{young}
\usepackage{multicol}
\usepackage{mathrsfs}
\usepackage[all]{xy}
\usepackage{tikz}
\usepackage{mathtools}
\usepackage{tabularx}
\usepackage{array}
\usepackage{commath}
\usepackage{ytableau}

\usepackage{scrextend}

\usepackage{hyperref}

\theoremstyle{plain}
\theoremstyle{definition}

\newtheorem{theorem}{Theorem}[section]

\newtheorem{definition}[theorem]{Definition}

\newtheorem{question}[theorem]{Question}
\newtheorem{example}[theorem]{Example}
\newtheorem{proposition}[theorem]{Proposition}
\newtheorem{corollary}[theorem]{Corollary}

\DeclareMathAlphabet{\mathpzc}{OT1}{pzc}{m}{it}

\newcommand{\s}{\sigma}

\newcommand{\supp}{\textnormal{\textsf{supp}}}
\newcommand{\mf}[1]{\mbox{$\mathfrak #1$}}
\newcommand{\rw}[1]{\left[{#1}\right]}

\newcommand{\newperm}[2]{{#1}^{[{#2}]}}

\begin{document}

\title[Boolean intersection ideals]{Boolean intersection ideals of permutations\\
in the Bruhat order}

\date{}

\author{Bridget Eileen Tenner}
\address{Department of Mathematical Sciences, DePaul University, Chicago, IL, USA}
\email{bridget@math.depaul.edu}
\thanks{Research partially supported by NSF Grant DMS-2054436 and Simons Foundation Collaboration Grant for Mathematicians 277603.}

\keywords{}%

\makeatletter
\@namedef{subjclassname@2020}{%
  \textup{2020} Mathematics Subject Classification}
\makeatother
\subjclass[2020]{Primary: 20F55; 
Secondary: 06A07, 
05E16
}

\begin{abstract}%
Motivated by recent work with Mazorchuk, we characterize the conditions under which the intersection of two principal order ideals in the Bruhat order is boolean. That characterization is presented in three versions: in terms of reduced words, in terms of permutation patterns, and in terms of permutation support. The equivalence of these properties follows from an analysis of what it means to have a specific letter repeated in a permutation's reduced words; namely, that a specific $321$-pattern appears.
\end{abstract}

\maketitle

In recent work with Mazorchuk, we studied intersections of a boolean principal order ideal with an arbitrary principal order ideal in the Bruhat order of the symmetric group \cite{mazorchuk tenner}. That boolean requirement was enough to guarantee that the grade of the simple module indexed by a boolean element is given by Lusztig's $\mathbf{a}$-function (see \cite{lusztig}). In the present work, we study a related question; namely, given two arbitrary principal order ideals in the Bruhat order, when is their intersection boolean? Certainly if either permutation itself is boolean then their intersection will also be boolean, but that is merely a special case. We answer the general question by first characterizing a more general property relating reduced words and permutation patterns, which is related, in some ways, to previous work \cite{tenner rdpp, tenner rwm}.

We begin this note with definitions and a presentation of the problem. In Section~\ref{sec:previous}, we use previous work to give our first characterization of boolean intersection ideals in terms of an ``interlacing'' property of reduced words (Theorem~\ref{thm:not boolean means braid}). Section~\ref{sec:interlaced} explores the interlacing property more deeply in its own right, and shows that interlacing in $w$ is equivalent to a particular $321$-pattern appearing in a permutation $\newperm{w}{k}$ (Theorem~\ref{thm:interlaced 321}). We conclude with Section~\ref{sec:answer}, giving a pattern characterization in Corollary~\ref{cor:boolean intersection and centering} for permutations $v$ and $w$ that is equivalent to $B(v) \cap B(w)$ being boolean, as well as a characterization in terms of support (Corollary~\ref{cor:boolean intersection and support}). In Sections~\ref{sec:previous} and~\ref{sec:answer}, we also address features that one might hope to have in these boolean intersections, but which do not always hold.

\section{Definitions and notation}\label{sec:definitions}

This work is concerned with permutations in the symmetric group $\mf{S}_n$, under the Bruhat order. For any $w \in \mf{S}_n$, we will write $B(w)$ for the principal order ideal of $w$. Our interest is in intersections of the form
$$B(v) \cap B(w),$$
which have not previously received much attention. By the subword property (see \cite{bjorner brenti}), this intersection is an order ideal.

The Coxeter group $\mf{S}_n$ is generated by the adjacent transpositions $\{\s_i : i \in [1,n-1]\}$. As in \cite{tenner rdpp, tenner rwm}, we will be interested in both the one-line representation of a permutation $w \in \mf{S}_n$ and in the reduced words $R(w)$ for $w$. To indicate that a string $s$ of values represents a reduced word, and not a permutation in one-line notation, we will write $\rw{s}$. Permutations are composed from right to left, and so
$$R(2431) = \{\rw{1232}\!, \rw{1323}\!, \rw{3123}\},$$
and we can write $2431 = \rw{1232} = \rw{1323} = \rw{3123}$. The example of $R(2431)$ includes an important feature that we highlight with a definition.

\begin{definition}\label{defn:interlace}
Consider a permutation $w \in \mf{S}_n$ whose reduced words contain both $k$ and $k+1$ for some $k \in [1,n-2]$. If $w$ has a reduced word with one (or both) of the forms
\begin{equation}\label{eqn:interlace defn}
\rw{\cdots k \cdots (k+1) \cdots k \cdots} \hspace{.25in} \text{or} \hspace{.25in} \rw{\cdots (k+1) \cdots k \cdots (k+1) \cdots},
\end{equation}
then $k$ and $k+1$ are \emph{interlaced} in $w$.
\end{definition}

Thus $2431$ interlaces $2$ and $3$, and it does not interlace $1$ and $2$.

Note that the hypothesis on ``a reduced word'' in Definition~\ref{defn:interlace} can equivalently be replaced by the same hypothesis on \emph{all} reduced words of $w$. This is because any two reduced words for $w$ are related to each other by a sequence of commutation and braid moves.

As studied previously, a permutation is \emph{boolean} if its principal order ideal in the Bruhat order is isomorphic to a boolean algebra \cite{hultman vorwerk, mazorchuk tenner, ragnarsson tenner homotopy, ragnarsson tenner homology, tenner patt-bru}. Here we extend that definition to order ideals.

\begin{definition}\label{defn:boolean order ideal}
An order ideal in the Bruhat order is \emph{boolean} if all of its elements are boolean elements.
\end{definition}

The principal order ideal of any boolean element is certainly boolean. Perhaps more interestingly, the $5$-element order ideal consisting of the permutations $B(2143) \cup B(1324) = \{1234,2134,1324, 1243,2143\} \subset \mf{S}_4$ is also boolean. This order ideal is depicted in Figure~\ref{fig:sample boolean ideal}.

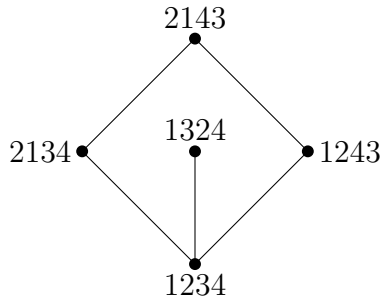
\begin{figure}[htbp]
\begin{tikzpicture}[scale=.75]
\foreach \x in {(0,0), (0,2), (2,2), (-2,2), (0,4)} {\fill \x circle (3pt);}
\draw (0,0) -- (-2,2) -- (0,4) -- (2,2) -- (0,0) -- (0,2);
\draw (0,0) node[below] {$1234$};
\draw (-2,2) node[left] {$2134$};
\draw (2,2) node[right] {$1243$};
\draw (0,4) node[above] {$2143$};
\draw (0,2) node[above] {$1324$};
\end{tikzpicture}
\caption{The boolean order ideal $B(2143) \cup B(1324) \subset \mf{S}_4$.}\label{fig:sample boolean ideal}
\end{figure}

The purpose of this note is to answer the following question.

\begin{question}\label{ques:boolean order ideal}
Under what circumstances is $B(v) \cap B(w)$ a boolean order ideal?
\end{question}

\section{Characterization in terms of interlacing}\label{sec:previous}

Two previous results will be key to answering Question~\ref{ques:boolean order ideal}. The first of these is a characterization of boolean permutations. Although we state this in terms of permutations and the symmetric group, a version of this characterization exists for any Coxeter group.

\begin{proposition}[{\!\!\cite{tenner patt-bru}}]\label{prop:boolean characterization}
A permutation is boolean if and only if its reduced words contain no repeated letters. This is equivalent to the permutation avoiding the patterns $321$ and $3412$.
\end{proposition} 

Fix permutations $v, w \in \mf{S}_n$. Determining whether $B(v) \cap B(w)$ is boolean amounts to checking each $u \in B(v) \cap B(w)$ against the equivalent conditions of Proposition~\ref{prop:boolean characterization}: the intersection $B(v) \cap B(w)$ fails to be boolean if and only if it contains some $u$ whose reduced words contain repeated letters. That is, the question amounts to determining whether such a $u$ has a reduced word containing two copies of some $k \in [1,n-1]$. By the subword property, then, we can say the following.

\begin{theorem}\label{thm:not boolean means braid}
The intersection $B(v) \cap B(w)$ is not boolean if and only if $\rw{k(k+1)k} \in B(v) \cap B(w)$ for some $k \in [1,n-2]$. Equivalently, the intersection ideal $B(v) \cap B(w)$ fails to be boolean if and only if there exists $k \in [1,n-2]$ for which $k$ and $k+1$ are interlaced in both $v$ and $w$.
\end{theorem}

While this does, in a sense, answer Question~\ref{ques:boolean order ideal}, we can clarify that answer further. In particular, previous work has shown connections between reduced words and permutation patterns, and we can make use of those relationships here. To do so, we will call upon the following result.

\begin{proposition}[{\!\!\cite{mazorchuk tenner}}]\label{prop:interlaced characterization}
Fix a permutation $w \in \mf{S}_n$. Then $k$ and $k+1$ are interlaced in $w$ if and only if there exists $i \in [1,k]$ and $j \in [k+2,n]$ such that $w(i) > k+1$ and $w(j) < k+1$.
\end{proposition}

We find it illuminating here to also highlight a red herring. In \cite{tenner rep-range}, we introduced the language of straddling patterns to describe particular occurrences of $321$- and $3412$-patterns, and their relationship to the number of times each letter $k \in [1,n-1]$ can appear in reduced words of a permutation $w \in \mf{S}_n$. A particular result from that work, namely \cite[Theorem 3.3]{tenner rep-range}, looks at first glance like it might be helpful for characterizing permutations for which $B(v) \cap B(w)$ is boolean. However, a subtlety of $3412$-pattern containment dashes those hopes. For example, there is no $k$ for which the permutations $34125, 14523 \in \mf{S}_5$ both straddle $k$ in both position and value. However, $B(34125) \cap B(14523)$ is not boolean because it contains $B(\rw{232})$. The hiccup here is that $R(34125) = \{\rw{2132},\rw{2312}\}$ and $R(14523) = \{\rw{3243},\rw{3423}\}$, and none of these contains both $232$ and $323$ as subwords.

\section{The meaning of interlaced letters}\label{sec:interlaced}

Our goal now is to understand what it means for letters to be interlaced in a permutation. Because interlacing is defined by the configurations in \eqref{eqn:interlace defn} and because $R(321) = \{\rw{121},\rw{212}\}$, it is tempting to expect a relationship between interlacing and $321$-patterns. That intuition is correct, but the relationship is not straightforward. For an indication of why not, recall the interlacing in the $321$-avoiding permutation $3412$ discussed above.

As we will show, interlaced letters in $w$ will imply the presence of a $321$-pattern in a permutation $\newperm{w}{k}$. This $\newperm{w}{k}$ will be related to $w$, but will fix $k+1$ and will change the rest of $w$ as little as possible while being careful about which inversions are created by the changes that are necessary. Moreover, we will show that interlacing in $w$ is equivalent to a specific and describable $321$-pattern in the permutation $\newperm{w}{k}$. 

\begin{definition}\label{defn:w'}
Fix a permutation $w \in \mf{S}_n$ and a value $k \in [1,n-2]$. Set $m := w^{-1}(k+1)$ and define $\newperm{w}{k}$ as follows.
\begin{itemize}
\item If $m = k+1$, then $\newperm{w}{k} := w$.
\item If $m > k+1$ and $w(k+1) > k+1$, or if $m < k+1$ and $w(k+1) < k+1$, then let $\newperm{w}{k} \in \mf{S}_n$ be the permutation defined by
$$\newperm{w}{k}(i) := \begin{cases}
w(i) & \text{ for } i \not\in \{m,k+1\},\\
w(m) = k+1 & \text{ for } i = k+1, \text{ and}\\
w(k+1) & \text{ for } i = m.
\end{cases}$$
\item If $m > k+1$ and $w(k+1) < k+1$, then there is necessarily some $t < k+1$ with $w(t) > k+1$, and we (arbitrarily) pick the maximal such $t$. If $m < k+1$ and $w(k+1) > k+1$, then there is necessarily some $t > k+1$ with $w(t) < k+1$, and we (arbitrarily) pick the minimal such $t$. In either case, let $\newperm{w}{k} \in \mf{S}_n$ be the permutation defined by
$$\newperm{w}{k}(i) := \begin{cases}
w(i) & \text{ for } i \not\in \{m,k+1,t\},\\
w(m) = k+1 & \text{ for } i = k+1,\\
w(k+1) & \text{ for } i = t, \text{ and}\\
w(t) & \text{ for } i = m.
\end{cases}$$
\end{itemize}
\end{definition}

Note what changes between $w$ and $\newperm{w}{k}$ in each case of Definition~\ref{defn:w'}. In the first case, nothing changes. In the second, the values $k+1$ and $w(k+1)$ create an inversion in $w$ and are then swapped in the one-line notation to form $\newperm{w}{k}$. The third scenario is when the values $k+1$ and $w(k+1)$ do not form an inversion, and so we find a value $w(t)$ that forms an inversion with both of them. In the first subcase, the permutation $w$ has a $312$-pattern in positions $t < k+1 < m$, and the values in those positions get permuted to form a $123$-pattern in $\newperm{w}{k}$. In the second subcase, the permutation $w$ has a $231$-pattern in positions $m < k+1 < t$, and the values in those positions get permuted to form a $123$-pattern in $\newperm{w}{k}$. In all situations, $k+1$ is a fixed point of the permutation $\newperm{w}{k}$.

This permutation $\newperm{w}{k}$ is what will contain the designated $321$-pattern when $k$ and $k+1$ are interlaced in $w$. The particularity of that pattern is defined by how it occurs in the permutation.

\begin{definition}\label{defn:centered 321}
Let $w$ be a permutation fixing a value $h$. If $w$ contains a $321$-pattern with middle value equal to (and appearing in position) $h$, then we will say that $w$ has a $321$-pattern \emph{centered at $h$}.
\end{definition}

Having defined $\newperm{w}{k}$ and centering, we can now describe interlaced letters in terms of $321$-patterns.

\begin{theorem}\label{thm:interlaced 321}
A permutation $w$ interlaces $k$ and $k+1$ if and only if the permutation $\newperm{w}{k}$ contains a $321$-pattern centered at $k+1$.
\end{theorem}

\begin{proof}
By Proposition~\ref{prop:interlaced characterization}, we have that $w$ interlaces $k$ and $k+1$ if and only if there exist $i$ and $j$ such that
\begin{equation}\label{eqn:inequalities for interlacing}
i < k+1 < j \hspace{.5in} \text{and} \hspace{.5in} w(i) > k+1 > w(j).
\end{equation}
These inequalities bear a notable resemblance to the definition of a $321$-pattern, but the middle term in the latter set is not quite what one would need unless $w$ were to fix $k+1$. 

As described in Definition~\ref{defn:w'}, the permutation $\newperm{w}{k}$ is constructed from $w$ by, among other things, moving the value $k+1$ into position $k+1$.

If $w(k+1) = k+1$, then $w = \newperm{w}{k}$ and the inequalities in \eqref{eqn:inequalities for interlacing} would describe the desired occurrence of $321$ in $\newperm{w}{k}$.

Now assume that $w(k+1) \neq k+1$. Thus $w \neq \newperm{w}{k}$, and we recall the definitions of $m$ and $t$ from Definition~\ref{defn:w'}. In particular, there are no inversions among the positions $\{k+1,m,t\}$ (or just $\{k+1,m\}$ in the second category of the definition) in $\newperm{w}{k}$. Therefore there is a $321$-pattern in $\newperm{w}{k}$ in positions $i < k+1 < j$ and having values $\newperm{w}{k}(i) > \newperm{w}{k}(k+1) = k+1 > \newperm{w}{k}(j)$ if and only if neither $i$ nor $j$ is equal to $m$ or $t$ (if $t$ is defined). That is, there is such a $321$-pattern if and only if $\newperm{w}{k}(i) = w(i)$ and $\newperm{w}{k}(j) = w(j)$. Thus $i$ and $j$ satisfy the inequalities of \eqref{eqn:inequalities for interlacing}, and so there is such a $321$-pattern in $\newperm{w}{k}$ if and only if $k$ and $k+1$ are interlaced in $w$.
\end{proof}

It is helpful to demonstrate Definition~\ref{defn:w'} and Theorem~\ref{thm:interlaced 321} with examples.

\begin{example}
Consider $w = 462135 \in \mf{S}_6$.
\begin{enumerate}[label=(\alph*)]
\item If $k = 1$, then we are in the second category of Definition~\ref{defn:w'}, meaning that $\newperm{w}{1} = 426135$. There is a $321$-pattern in positions $1 < 2 < 4$ of this permutation, centered at $2 = k + 1$. This confirms the fact that $w$ interlaces $1$ and $2$, as we see, for example, in the reduced word $\rw{53412312} \in R(w)$.
\item For $k = 2$, we are in the third category of the definition and $t = 2$. Then $\newperm{w}{2} = 423165$, which has the desired $321$-pattern in positions $1 < 3 < 4$, centered at $3 = k+1$. This confirms the fact that $w$ interlaces $2$ and $3$, as we see in $\rw{53412312} \in R(w)$.
\item With $k = 4$, we are again in the third category with $t = 2$. Then $\newperm{w}{4} = 432156$, which has no $321$-pattern centered at $5 = 4+1$. This confirms the fact that $w$ does not interlace $4$ and $5$, as we see from its reduced words $\rw{53412312}$, and so on.
\end{enumerate}
\end{example}

\section{Characterizations in terms of patterns and support}\label{sec:answer}

Theorem~\ref{thm:interlaced 321} builds off of results like Proposition~\ref{prop:boolean characterization} and \cite[Theorem~2.1]{billey jockusch stanley} from the literature. The last of these shows that being $321$-avoiding is equivalent to having no consecutive substring $k(k+1)k$ in any reduced words. In Proposition~\ref{prop:boolean characterization}, having all distinct letters prevents both $321$- and $3412$-patterns. In the result of this paper, on the other hand, repeating the letter $k$ and having, without loss of generality, a $k+1$ between the repeated letters, forces a $321$-pattern in $\newperm{w}{k}$ centered at $k+1$. We can use this language to expand upon the result of Theorem~\ref{thm:not boolean means braid}, characterizing boolean intersection ideals $B(v) \cap B(w)$ in terms of patterns and giving a more complete answer to Question~\ref{ques:boolean order ideal}.

\begin{corollary}\label{cor:boolean intersection and centering}
Fix permutations $v, w \in \mf{S}_n$. The intersection $B(v) \cap B(w)$ is boolean if and only if, for all $k \in [1,n-2]$, at most one of the permutations $\newperm{v}{k}$ and $\newperm{w}{k}$ has a $321$-pattern centered at $k+1$.
\end{corollary}

We demonstrate this characterization using boolean and non-boolean examples.

\begin{example}\label{ex:computing intersections}
Consider the principal order ideal of $w = 5137246$, intersected with the principal order ideals of two different permutations: $u = 5213674$ and $v = 3512674$.
\begin{enumerate}[label=(\alph*)]
\item To determine the structure of $B(u) \cap B(w)$, we compute the values shown in Table~\ref{table:u and w}. 
\begin{table}[htbp]
$$\begin{array}{c||c|c||c|c}
& \newperm{u}{k}, \text{ permuted} & \newperm{w}{k}, \text{ permuted} & 321\text{-pattern in } \newperm{u}{k} & 321\text{-pattern in } \newperm{w}{k}\\
\raisebox{0in}[.1in][.1in]{$k$}  & \text{letters in {\color{red} red}} & \text{letters in {\color{red} red}} & \text{centered at } k+1? & \text{centered at } k+1?\\
\hline
\hline
\raisebox{0in}[.2in][.1in]{$1$} & 5213674 & {\color{red} 1}{\color{red} 2}37{\color{red} 5}46 & \text{yes} & \text{no} \\
\hline
\raisebox{0in}[.2in][.1in]{$2$} & {\color{red} 1}2{\color{red} 3}{\color{red} 5}674 & 5137246 & \text{no} & \text{yes}\\
\hline
\raisebox{0in}[.2in][.1in]{$3$} & {\color{red} 3}21{\color{red} 4}67{\color{red} 5} & 513{\color{red} 4}2{\color{red} 7}6 & \text{no} & \text{yes}\\
\hline
\raisebox{0in}[.2in][.1in]{$4$} & {\color{red} 4}213{\color{red} 5}7{\color{red} 6} & {\color{red} 2}137{\color{red} 5}46 & \text{no} & \text{yes}\\
\hline
\raisebox{0in}[.2in][.1in]{$5$} & 5213{\color{red} 4}{\color{red} 6}{\color{red} 7} & 513{\color{red} 4}2{\color{red} 6}{\color{red} 7} & \text{no} & \text{no}
\end{array}$$
\caption{Data for $u = 5213674$ and $w = 5137246$.}\label{table:u and w}
\end{table}
Each row of the table has at least one ``no'' in its last two columns, so the intersection $B(u) \cap B(w)$ is boolean. We can confirm this by computing reduced words, such as $u = \rw{4321256}$ and $w = \rw{64323154}$, and using Theorem~\ref{thm:not boolean means braid}. In fact, $B(u) \cap B(w)$ is the union of two boolean principal order ideals, $B(\rw{43215}) \cup B(\rw{64321})$.
\item To determine the structure of $B(v) \cap B(w)$, we compute the values shown in Table~\ref{table:v and w}. 
\begin{table}[htbp]
$$\begin{array}{c||c|c||c|c}
& \newperm{v}{k}, \text{ permuted} & \newperm{w}{k}, \text{ permuted} & 321\text{-pattern in } \newperm{v}{k} & 321\text{-pattern in } \newperm{w}{k}\\
\raisebox{0in}[.1in][.1in]{$k$}  & \text{letters in {\color{red} red}} & \text{letters in {\color{red} red}} & \text{centered at } k+1? & \text{centered at } k+1?\\
\hline
\hline
\raisebox{0in}[.2in][.1in]{$1$} & 3{\color{red} 2}1{\color{red} 5}674 & {\color{red} 1}{\color{red} 2}37{\color{red} 5}46 & \text{yes} & \text{no} \\
\hline
\raisebox{0in}[.2in][.1in]{$2$} & {\color{red} 1}5{\color{red} 3}2674 & 5137246 & \text{yes} & \text{yes}\\
\hline
\raisebox{0in}[.2in][.1in]{$3$} & 3{\color{red} 2}1{\color{red} 4}67{\color{red} 5} & 513{\color{red} 4}2{\color{red} 7}6 & \text{no} & \text{yes}\\
\hline
\raisebox{0in}[.2in][.1in]{$4$} & 3{\color{red} 4}12{\color{red} 5}7{\color{red} 6} & {\color{red} 2}137{\color{red} 5}46 & \text{no} & \text{yes}\\
\hline
\raisebox{0in}[.2in][.1in]{$5$} & 3512{\color{red} 4}{\color{red} 6}{\color{red} 7} & 513{\color{red} 4}2{\color{red} 6}{\color{red} 7} & \text{no} & \text{no}
\end{array}$$
\caption{Data for $v = 3512674$ and $w = 5137246$.}\label{table:v and w}
\end{table}
The row for $k = 2$ has ``yes'' in both of the last two columns, meaning that $B(v) \cap B(w)$ is not boolean. Indeed, $1432567 = \rw{232} \in B(v) \cap B(w)$, and so the non-boolean poset $B(\rw{232})$ is a subset of this intersection ideal.
\end{enumerate}
\end{example}

Because $k+1$ is a fixed point of $\newperm{w}{k}$, we can use \cite[Lemma~2.8]{tenner rep-patt} to frame boolean intersection ideals in one more light, now in terms of ``support.''

\begin{definition}\label{defn:support}
The \emph{support} of a permutation $w \in \mf{S}_n$ is the set $\supp(w) \subseteq [1,n-1]$ consisting of all letters that appear in reduced words for $w$.
\end{definition}

We can use \cite[Lemma~2.8]{tenner rep-patt} and Corollary~\ref{cor:boolean intersection and centering} to characterize boolean intersection ideals by support. As a starting point, note that if $u$ fixes $k+1$, then the lemma implies that $k \in \supp(u)$ if and only if $k+1 \in \supp(u)$.

\begin{corollary}\label{cor:boolean intersection and support}
For permutations $v,w \in \mf{S}_n$, the following statements are equivalent:
\begin{itemize}
\item the intersection $B(v) \cap B(w)$ is boolean;
\item $k$ is in the support of at most one of $\newperm{v}{k}$ and $\newperm{w}{k}$, for all $k \in [1,n-2]$;
\item $k+1$ is in the support of at most one of $\newperm{v}{k}$ and $\newperm{w}{k}$, for all $k \in [1,n-2]$; and
\item $\{k,k+1\} \cap \supp(u) = \emptyset$ for at least one $u \in \{\newperm{v}{k},\newperm{w}{k}\}$, for all $k \in [1,n-2]$.
\end{itemize}
\end{corollary}

We conclude this note with another red herring. From Theorem~\ref{thm:not boolean means braid} and \cite[Lemma~5.10]{mazorchuk tenner}, one might hope that when an intersection $B(v) \cap B(w)$ is known to be boolean, perhaps that intersection is equal to $B(v') \cap B(w')$ for some permutations $v'$ and $w'$ that are themselves boolean. Sadly this is not always the case, as we can see with $v = \rw{123} = 2341$ and $w = \rw{2132} = 3412$. The intersection of their principal order ideals is equal to $B(v) \setminus \{v\}$, and yet any boolean $w'$ whose principal order ideal contains both $\rw{12}$ and $\rw{23}$ will necessarily also contain $\rw{123} = v$, in which case the intersection ideal would be all of $B(v)$. This example is depicted in Figure~\ref{fig:cannot reduce to boolean}.

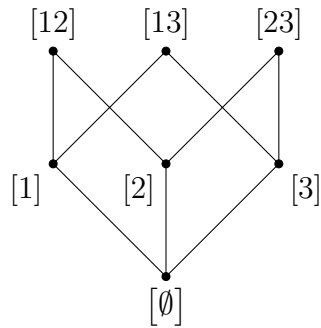
\begin{figure}[htbp]
\begin{tikzpicture}[scale=.75]
\filldraw (0,0) circle (2pt) coordinate (empty) node[below] {$\rw{\emptyset}$};
\filldraw (-2,2) circle (2pt) coordinate (1) node[below left] {$\rw{1}$};
\filldraw (0,2) circle (2pt) coordinate (2) node[below left] {$\rw{2}$};
\filldraw (2,2) circle (2pt) coordinate (3) node[below right] {$\rw{3}$};
\filldraw (-2,4) circle (2pt) coordinate (12) node[above] {$\rw{12}$};
\filldraw (0,4) circle (2pt) coordinate (13) node[above] {$\rw{13}$};
\filldraw (2,4) circle (2pt) coordinate (23) node[above] {$\rw{23}$};
\foreach \x in {1,2,3} {\draw (empty) -- (\x);}
\draw (1) -- (12) -- (2) -- (23) -- (3) -- (13) -- (1);
\end{tikzpicture}
\caption{The boolean intersection ideal $B(\rw{123}) \cap B(\rw{2132})$.}\label{fig:cannot reduce to boolean}
\end{figure}

\end{document}